\documentclass[11pt,a4paper]{article}
\usepackage{mathrsfs}
\usepackage{amsmath,amsthm,color,graphicx,eepic}
\usepackage{amssymb} \usepackage{verbatim}
\usepackage{dsfont}
\usepackage{graphicx}
\usepackage{subcaption}

\usepackage{hyperref}
\usepackage{xcolor}
\usepackage[normalem]{ulem}
\usepackage{cancel}

\newif\ifshowchanges
\newcommand{\added}[1]{\ifshowchanges{{\color{red}#1}}\else#1\fi}
\newcommand{\deleted}[1]{\ifshowchanges{{\color{red}\sout{#1}}}\else\relax\fi}
\newcommand{\replaced}[2]{\ifshowchanges{{\color{red}\sout{#1}} {\color{red}#2}}\else#2\fi}

\newcommand{\mreplaced}[2]{\ifshowchanges{{\color{red}\cancel{#1}\,#2}}\else#2\fi}

\usepackage{mathtools} %coloneqq

\newtheorem{theorem}{Theorem}%[section]

\newtheorem{conjecture}[theorem]{Conjecture}

\newtheorem{claim}{Claim}
\newtheorem*{claim*}{Claim}

\newcommand{\ex}{{\rm ex}}

\newcommand{\floor}[1]{\left\lfloor{#1}\right\rfloor}

\newcommand{\abs}[1]{\left\vert{#1}\right\vert}

\begin{document}

\title{Turán-Type Extremal Results for Distance-$k$ Graphs}
\author{
Zhen He\thanks{School of Mathematics and Statistics, Beijing Jiaotong University, Beijing, China.} \and
Nika Salia\thanks{Theoretical Computer Science Department, Faculty of Mathematics and Computer Science, Jagiellonian University, Krak\'{o}w, Poland.} \and
Casey Tompkins\thanks{HUN-REN Alfr\'ed R\'enyi Institute of Mathematics} \and
Xiutao Zhu\thanks{School of Mathematics, Nanjing University of Aeronautics and Astronautics, Nanjing, China.} \and
}
\date{}
\maketitle

\begin{abstract}
We study Tur\'an-type extremal problems for distance graphs, motivated by work of Csikv\'ari, Bollob\'as, Tyomkyn, and Uzzell.
We determine the maximum number of vertex pairs at distance three in an $n$-vertex graph with no triangle formed by these pairs, resolving the first case of a conjecture of Tyomkyn and Uzzell. We also determine the maximum number of vertex pairs at distance two in an $n$-vertex graph with no triangle formed by these pairs and give a complete characterization of the extremal graphs, settling another problem of Tyomkyn and Uzzell.
\end{abstract}

\section{Introduction}

Tur\'an's theorem determines the maximum number of edges in an \(n\)-vertex graph containing no clique of a fixed size. In this paper we study the corresponding problem in which edges are replaced by pairs of vertices at a prescribed distance. Given a graph \(G\) and an integer \(k \ge 1\), the \(k\)-distance graph \(G_k\) is the graph on \(V(G)\) in which two vertices are adjacent if and only if they are at distance exactly \(k\) in \(G\). Thus \(|E(G_k)|\) is the number of unordered pairs of vertices at distance \(k\) in \(G\).

This problem is related to earlier work on walks and paths in trees. Csikv\'ari~\cite{csikvari2010poset} proved that among \(n\)-vertex trees the star and the path maximize and minimize, respectively, the number of closed walks of length \(k\). Bollob\'as and Tyomkyn~\cite{bollobas2012walks} later gave another proof and extended the result to all walks of length \(k\). They also considered the problem of maximizing the number of paths of length \(k\) in an \(n\)-vertex tree, showing that the extremal examples belong to a family of trees they called \(t\)-brooms. Since in a tree every path of length \(k\) is determined by its endpoints, this is equivalent to maximizing the number of pairs of vertices at distance \(k\).

Tyomkyn and Uzzell~\cite{tyomkyn2013turan} initiated the corresponding problem for general graphs. More generally, for a fixed graph \(H\), define
\[
\operatorname{ex}_k(n,H):=\max\{|E(G_k)|:\ |V(G)|=n \text{ and } G_k \text{ is } H\text{-free}\}.
\]
The graphs $G$ for which $|E(G_k)|=\operatorname{ex}_k(n,H)$ are referred to as extremal graphs.
They proved that when \(k=2\), the star is an extremal graph. For \(k \ge 3\) and \(n\) sufficiently large, they conjectured that the maximum is attained by a \(t\)-broom for an appropriate choice of \(t\). For even \(k \ge 4\), a \(t\)-broom consists of a central vertex joined to one endpoint of each of \(t\) paths on \((k-2)/2\) vertices, with an arbitrary number of leaves attached to the other endpoint of each path. For odd \(k \ge 3\), a \(t\)-broom consists of a clique of size \(t\), where each clique vertex is joined to one endpoint of a path on \((k-3)/2\) vertices, again with leaves attached to the other endpoint of each path.

We consider here the problem of maximizing \(|E(G_k)|\) under a bound on the clique number of the distance graph \(G_k\). A particularly natural case is the triangle-free distance-two problem. Tyomkyn and Uzzell~\cite{tyomkyn2013turan} constructed an \(n\)-vertex graph \(G\) such that \(G_2\) is triangle-free and
$|E(G_2)|=\left\lfloor \frac{(n-1)^2}{4}\right\rfloor+1,$
and asked whether this value is the best possible. We \replaced{positivly}{positively} answer this question and determine all extremal graphs.

\begin{theorem}\label{G2}
For all $n \ge 5$,
\[
\ex_2(n,K_3) = \floor{ \frac{(n-1)^2}{4} } + 1.
\]
The extremal graphs are the complements of graphs obtained from a Tur\'an graph  $T(n-1,2)$ with bipartition $(A,B)$ 
by fixing a vertex $a_1 \in A$ and a non-trivial partition $B = B' \cup B''$, 
removing the edges between $a_1$ and $B'$, and introducing a new vertex $a_2$ 
with the neighbourhood $\{a_1\} \cup B'$. 
\end{theorem}

Going beyond distance $2$-graphs, Tyomkyn and Uzzell~\cite{tyomkyn2013turan} considered the problem of maximizing $|E(G_k)|$ when $G_k$ contains no clique of size $t+1$.

\begin{conjecture}[Tyomkyn--Uzzell~\cite{tyomkyn2013turan}]
Let $k\ge 3$ and $t\ge 2$. There is a function $h(k,t)$ such that if $n\ge h(k,t)$, then the maximum of $|E(G_k)|$ over $n$-vertex graphs $G$ with $\omega(G_k)\le t$ is attained when $G_k$ is isomorphic to the $k$-distance graph of a $p$-broom.
\end{conjecture}

As a step in proving this conjecture, they gave the following sharp result in the case of $t=2$ for $k$ and $n$ \replaced{taken to be sufficiently large}{sufficiently large}.
\begin{theorem}[Tyomkyn--Uzzell~\cite{tyomkyn2013turan}]\label{TU}
For $k$ sufficiently large and $n$ sufficiently large as a function of $k$, every $n$-vertex graph $G$ such that $G_k$ is triangle-free satisfies
\[
|E(G_k)| \le \frac{(n-k+1)^2}{4}.
\]
Moreover, if equality holds, then $G_k$ is isomorphic to $D_k$, where $D$ is a double broom.
\end{theorem}

They made the following conjecture, implying that the conclusion of Theorem~\ref{TU} holds for $k=3$ and $n\ge 9$.

\begin{conjecture}[Tyomkyn--Uzzell~\cite{tyomkyn2013turan}]\label{conjTU}
For $k\ge 3$ and $n\ge k+1$ except when $k=3$ and $n=8$,  every $n$-vertex graph $G$ such that $G_k$ is triangle-free satisfies
\[
|E(G_k)| \le \frac{(n-k+1)^2}{4}.
\]
Moreover, if equality holds, then $G_k$ is isomorphic to $D_k$, where $D$ is a balanced double broom.
\end{conjecture}
%casey - we should mention what happens in the (3,8) case.

Li, Ma, Shi and Yue~\cite{li2015note} showed that Conjecture~\ref{conjTU} for $k=3$ holds if one includes the additional assumption that the graph contains a vertex whose neighborhood is covered by at most two cliques.

%Consider the graph $H$ defined by a path $uvwz$ with $a$ pendant edges adjacent to $w$ and $b$ pendant edges adjacent to $z$ so that $a+b=n+4$. $H_3$ has $(a+1)(b+1)$ edges but only one isolated vertex, so it is not isomorphic to $D_3$.  Choosing $a$ and $b$ to be balanced, we obtain the required number of edges.  Further modifications of this construction are possible, so it appears that a complete equality characterization would be complicated.    

We prove that the upper bound in Conjecture~\ref{conjTU} holds for $k=3$ without additional assumptions. We further exhibit a broader class of extremal graphs, showing that the proposed characterization in Conjecture~\ref{conjTU} is incomplete. In particular, we construct a graph $H$ whose $3$-distance graph attains the bound, but for which $H_3$ is not isomorphic to $D_3$ for all $n\geq 9$.

\begin{theorem}\label{G3}
For $n\geq 18$,
\[
\ex_3(n,K_3) = \floor{\left( \frac{n-2}{2} \right)^2}.
\]
\end{theorem}

The proofs of Theorems~\ref{G2} and~\ref{G3} are given in Section~\ref{proofs}.

\section{Proofs of the main results}\label{proofs}
We will make use of the following result about the maximum size of a triangle-free \replaced{nonbipartite}{non-bipartite} graph, which appeared in the paper of Erd\H{o}s~\cite{erdos1962theorem} as a lemma\added{,} where it was attributed to Erd\H{o}s and Gallai, and independently \added{to }Andrásfai. In fact, we will make use of a stronger version appearing in a paper by Kang and Pikhurko~\cite{kang2005maximum}, which also characterizes the equality cases. 

\begin{theorem}[Erd\H{o}s--Gallai, Andrásfai~\cite{erdos1962theorem}, Kang--Pikhurko~\cite{kang2005maximum}]\label{thm:Triangle_free_non_bip}
Let $G$ be an $n$-vertex triangle-free graph that is not bipartite with $n \ge 5$. Then
\[
|E(G)| \le \left\lfloor \frac{(n-1)^2}{4} \right\rfloor + 1.
\]

Equality holds if and only if $G$ is constructed from a Tur\'an graph  $T(n-1,2)$ with bipartition $(A,B)$ 
by fixing a vertex $a_1 \in A$ and a non-trivial partition $B = B' \cup B''$, 
removing the edges between $a_1$ and $B'$, and introducing a new vertex $a_2$ 
with the neighbourhood $\{a_1\} \cup B'$. 

\end{theorem}

We now give the proofs of Theorems~\ref{G2} and~\ref{G3}.

\subsection{Triangle-Free Distance-Two Graphs}

\begin{proof}[Proof of Theorem~\ref{G2}]
Let $G$ be an $n$-vertex graph such that $G_2$ is $K_3$-free and
\[
|E(G_2)| \ge \floor{\frac{(n-1)^2}{4}}+1.
\]
Let $v\in V(G)$ be a vertex of maximum degree in $G_2$, and denote $d_{G_2}(v)=\Delta$.
Since $|E(G_2)| \le \frac{n\Delta}{2}$, we have 
\[
\Delta \ge \frac{n-1}{2}.
\]
Set $N=N_{G_2}(v)$, and note that $G_2$ is triangle-free\replaced{,}{, so} the induced subgraph $G_2[N]$ is independent.
We distinguish two cases: $G_2$ is bipartite, and $G_2$ is non-bipartite.

\noindent\textbf{Case 1. $G_2$ is bipartite with bipartition $(A,B)$.}
The set $N$ lies entirely in one bipartition class of $G_2$, say $A$.
If $|A| > |N| = \Delta$, then $|B|\le n-\Delta-1$ and 
\[
|E(G_2)| \le \Delta |B|\le\Delta (n-\Delta-1)
\le \Big\lfloor \frac{(n-1)^2}{4} \Big\rfloor,
\]
a contradiction.

If $A = N$, then every vertex $u \in N$ is \replaced{at the distance of two}{at distance two} from $v$ in $G$.
Therefore, for each such $u$ there exists a vertex 
$x_u \in V(G) \setminus (\{v\} \cup N)$ with $ux_u \in E(G)$.
In particular, $u$ and $x_u$ are not adjacent in $G_2$.
It follows that $d_{G_2}(u) \le n-1-\Delta.$
Summing over $A$, we obtain
\[
|E(G_2)| \le |A|(n-1-\Delta)
\le \Big(\frac{n-1}{2}\Big)^2,
\]
a contradiction.

\noindent\textbf{Case 2. $G_2$ is not bipartite.}
By Theorem~\ref{thm:Triangle_free_non_bip},
$|E(G_2)| \le \Big\lfloor \frac{(n-1)^2}{4} \Big\rfloor + 1,$
which yields the desired bound.
Moreover, by Theorem~\ref{thm:Triangle_free_non_bip}, we have a well-defined structure of $G_2$.
The graph $G_2$  has the following structure. 
Start with the Tur\'an graph $T(n-1,2)$ with bipartition $(A,B)$.
Fix a vertex $a_1 \in A$ and choose a non-trivial partition
$B = B' \cup B''$.
Remove all edges between $a_1$ and $B'$, and add a new vertex $a_2$
whose neighbourhood is exactly $\{a_1\} \cup B'$.\
Let $G'$ be the complement of $G_2$.
The resulting graph $G'$ has diameter $2$, and therefore its distance-two graph is precisely $G_2$.
It remains to show that no proper subgraph of $G'$ has the same
distance-two graph.

Suppose that $G'' \subset G'$ satisfies $(G'')_2 = G_2$.
Since $G_2$ contains all edges between $A \setminus \{a_1\}$ and
$B' \cup B''$, it follows that in $G''$ every vertex of
$A \setminus \{a_1\}$ must be adjacent to both $a_1$ and $a_2$,
and every vertex of $B'$ (respectively $B''$) must be adjacent to
$a_1$ (respectively $a_2$).
Furthermore, since the subgraphs of $G_2$ induced on
$A$, $B'$, and $B''$ are empty, the corresponding induced subgraphs in $G''$ must be complete.
Thus all edges inside $A$, inside $B'$, and inside $B''$
are present in $G''$.

Finally, because $a_1$ and the vertices of $B''$ are adjacent in $G_2$,
there must exist at least one edge between $B'$ and $B''$ in $G''$.
However, since $G_2$ contains no edges inside $B' \cup B''$,
this forces the induced subgraph on $B' \cup B''$ in $G''$
to be complete.
Consequently, all edges of $G'$ are present in $G''$,
and hence $G'' = G'$. This completes the proof.
\end{proof}

\subsection{Triangle-Free Distance-Three Graphs}

\begin{proof}[Proof of Theorem~\ref{G3}]

For the lower bound, consider the following construction. Take a star with $\floor{n/2}$ or $\floor{(n+1)/2}$ leaves and attach all remaining vertices to these leaves arbitrarily. Edges may be added arbitrarily among the vertices of degree one with the same neighbourhood; all resulting graphs have the same number of edges $\floor{\frac{(n-2)^2}{4}}$ in $G_3$, but are \replaced{not isomorphic to the same graph}{not all isomorphic}. 
Indeed, consider two extremal constructions in the case when $n$ is odd. The first is an almost balanced double broom, and the second is a spider with $(n-1)/2$ legs of length two. Their $3$-distance graphs have the same number of edges, but are not isomorphic. 
In the first case, the $3$-distance graph is a complete balanced bipartite graph on $n-2$ vertices and two isolated vertices. 
In the second case, it is obtained from a complete balanced bipartite graph on $n-1$ vertices by deleting a perfect matching and adding an isolated vertex.
These constructions show that $\ex_3(n,K_3) \ge \floor{\frac{(n-2)^2}{4}}$, and that graphs attaining this bound need not have isomorphic $3$-distance graphs.

Let $G$ be an $n$-vertex graph such that $G_3$ is triangle-free and
\[
|E(G_3)| \ge \Big\lfloor \frac{(n-2)^2}{4} \Big\rfloor + 1.
\]
If $\Delta(G_3) \le \frac{n-4}{2}$, then
$
|E(G_3)|
\le \floor{\frac{n}{2}\cdot \frac{n-4}{2}}
< \Big\lfloor \frac{(n-2)^2}{4} \Big\rfloor + 1,$
a contradiction. Hence we have 
\[
   \Delta\mreplaced{\coloneq}{\coloneqq}\Delta(G_3) \ge \frac{n-3}{2}. 
\]

Fix a vertex $v$ with $d_{G_3}(v)=\Delta$ and define
$N_0:=\{v\}$, and for $i\in [3]$ let $N_i$ denote the set of vertices at distance $i$ from $v$ in $G$.
Let $N_{\ge4}$ be the remaining vertices.
Set $n_i \mreplaced{\coloneq}{\coloneqq}|N_i|$, thus we have $n_0=1$, $n_3=\Delta$, $n_1,n_2 \geq 1$ and $n_0+n_1+n_2+n_3+n_{\geq 4}=n$.

Suppose that every vertex $u\in N_3$ has at least four vertices
outside $N_3$ at a distance of at most two in $G$.
Since $G_3$ is triangle-free, $G_3[N_3]$ is an independent set of vertices, thus by assumption each vertex from $N_3$ has degree at most $n-\Delta-4$ in $G_3$.
Using the maximum degree bound for vertices outside $N_3$, we obtain
\[
|E(G_3)|
\le \frac12\big(
\Delta(n-\Delta-4)
+(n-\Delta)\Delta
\big)
\le \Delta(n-\Delta-2)<
\Big\lfloor \frac{(n-2)^2}{4} \Big\rfloor + 1,
\]
a contradiction.

 Since each $u \in N_3$ has at least one vertex in $N_1$ and one in $N_2$
at distance at most two in $G$, these vertices are not adjacent to $u$ in $G_3$, \replaced{thus each vertex from}{so each vertex in} $N_3$ has degree at most $n-\Delta-2$.
Thus, there exists a vertex $u \in N_3$ with
\[
d_{G_3}(u) = n - \Delta - 2
\ \text{or} \
n - \Delta - 3.
\]

\replaced{At first, let us assume}{First, assume} there is a vertex $u\in N_3$ with $d_{G_3}(u) = n - \Delta - 2$.
Let $w_1 \in N_1$ and $w_2 \in N_2$ be the unique vertices
\replaced{at a distance less than three}{at distance at most two} from $u$. Note that $G[V(G)\setminus (N_3\cup \{w_1,w_2\})]$ is an \replaced{indepoendent}{independent} set of vertices.

Since $d_G(u,x)=3$ for every $x\in N_1\setminus\{w_1\}$, and the only vertices at distance at most two from $u$
are  in $N_3\cup \{w_1,w_2\}$, we have $xw_1\in E(G)$.
Moreover, $w_2$ has no neighbors in $N_2$, and its only neighbor in $N_1$ is $w_1$.
The vertex $u$ has no neighbours in $N_{\ge4}$.

Assign charges as follows: vertices in $V(G)\setminus (N_3\cup\{w_1,w_2\})$ receive charge $\Delta$,
vertices in $N_3$ receive charge $n-\Delta-2$,
and $w_1,w_2$ receive charge $0$.
Each charge is an upper bound on the corresponding degree in $G_3$,
except possibly for $w_1,w_2$.
We shall redistribute the charges so that, in the end,
the charge assigned to every vertex is an upper bound on its degree in $G_3$.
If this is achieved, we are done, since the desired bound \deleted{bounds}\added{then follows from} the total charge.

Suppose there exists $w_4 \in N_{\ge 4}$ with $d_G(w_4,w_2)=3$, and let $w_4xyw_2$ be a shortest path. 
If $x,y \in N_3$, then each of $x$ and $y$ has a vertex, namely $w_4$, at distance less than three in $N_{\geq 4}$, and hence can transfer one unit of charge to $w_1$ and $w_2$. 
If $x \in N_{\ge 4}$ and $y \in N_3$, then $y$ has two vertices, $x$ and $w_4$, at distance less than three in $N_{\geq 4}$, and hence can transfer one unit of charge to $w_1$ and $w_2$.
Similarly, if $w_4 \in N_{\ge 4}$ is at distance three from $w_1$ but not from $w_2$ (otherwise this case has already been handled), say via a shortest path $w_4xyw_1$, then $x \in N_3$ and $y \in N_2$, and $x$ can transfer one unit of charge to $w_1$.

Let $u' \in N_3$ be at distance three from $w_i$,
for some $i \in \{1,2\}$.
Then there exists a vertex $w_2' \in N_2 \setminus \{w_2\}$
such that $u'w_2' \in E(G)$. If $w_2'w_1 \in E(G)$, then $w_2'$ is within distance two of 
$w_1$, $w_2$, and $u'$, and thus 
$d_{G_3}(w_2') \le \Delta - 1$.
Accordingly, $w_2'$ transfers one unit of charge to $w_2$.

If $u'$ has a neighbour in $N_3$ that is adjacent to $w_2$,
then $d_G(u',w_2)=2$.
Hence $u'$ transfers one unit of charge to $w_1$,
since it has two vertices in $N_2$ at distance at most two in $G$.
Otherwise, we claim that $d_G(u',w_2) > 3$,
and in this case $u'$ transfers one unit of charge to $w_1$.
Indeed, suppose there exists a path $u'xyw_2$ of length three.
Then $x$ is not a descendant of $w_2$ (in the BFS tree generated from root $v$), and therefore
$x u \notin E(G)$.
We have $x \in N_3$ and $d_G(x,u) \le 3$.
Since $G_3$ is triangle-free and $ux \notin E(G)$
(every neighbour of $u$ is a descendant of $w_2$),
it follows that $d_G(u,x)=2$.
Consequently, $d_G(u,u') = 3$,
contradicting the assumption that $u' \in N_3$.

Thus, every charge bounds the corresponding degree in $G_3$, yielding the desired bound on $|E(G_3)|$.
From here we may assume $d_{G_3}(u) \le n-\Delta-3$ for all $u\in N_3$. 

There is a vertex $u\in N_3$ such that $d_{G_3}(u)=n-\Delta-3$. 
Let $X=N_{G_3}(u)$ and $Y=N_3$ and $V=V(G)\setminus(X\cup Y)=\{v_1,v_2,v_3\}$. 
Let $X_i= N_{G_3}(v_i)\cap X$ and $Y_i= N_{G_3}(v_i)\cap Y$ for any $i\in [3]$. 
Without loss of generality assume $|X_1|\ge |X_2|\ge |X_3|$.
Let $\ell$ be the number of edges in $G_3[V]$,
$\ell\leq 2$ as $G_3$ is triangle-free.

For every subset $A \subseteq [3]$, let $y_A$ (respectively $x_A$)
denote the number of vertices $v \in Y$ (respectively $v \in X$)
whose neighbourhood in $V$ is exactly $\{v_i : i \in A\}$.

We have 
\begin{equation}\label{Eq:X_1_large}
    |X_1|+|X_2|+|X_3|\ge \floor{ \frac{n+1}{2}}-\ell.
\end{equation}
Indeed,  
\[
\Big\lfloor \frac{(n-2)^2}{4} \Big\rfloor + 1 \leq |E(G_3)|\leq |Y|(n-\Delta-3)+\ell+|X_1|+|X_2|+|X_3|\leq \floor{\frac{n^2-6n+8}{4}}+\ell+|X_1|+|X_2|+|X_3|.
\]
By the AM--GM inequality since $\Delta\geq \frac{n-2}{2}$.
By rearranging, we get
\[
|X_1|+|X_2|+|X_3| \geq \floor{\frac{n+1}{2}}-\ell.
\]

\begin{claim}\label{Claim:2.1}

\begin{align*}
|E(G_3)| & \le  y_{\{1,2,3\}}\cdot\min\{-|X_1\cup X_2\cup X_3|+3, 0\}
 +\sum\limits_{1\le i<j\le 3}y_{\{i,j\}}\cdot\min\{-|X_i\cup X_j|+2, 0\}\\
 &+\sum\limits_{i\in[3]}y_{\{i\}}\cdot\min\{-|X_i|+1, 0\}+\Delta(n-\Delta-3)+|X_1|+|X_2|+|X_3|+\ell,
\end{align*}
and
\begin{align*}
|E(G_3)| & \le x_{\{1,2,3\}}\cdot\min\{-|Y_1\cup Y_2\cup Y_3|+3, 0\}+\sum\limits_{1\le i<j\le 3}x_{\{i,j\}}\cdot\min\{-|Y_i\cup Y_j|+2,0 \}\\
 &+\sum\limits_{i\in[3]}x_{\{i\}}\cdot\min\{-|Y_i|+1, 0\}+\Delta(n-\Delta-3)+|Y_1|+|Y_2|+|Y_3|+\ell.
\end{align*}
\end{claim}

\begin{proof}[Proof of Claim~\ref{Claim:2.1}]
We count the number of edges in $G_3$, considering that it is triangle-free. \replaced{At first}{First}, we count all edges incident on vertices in $Y$.
\begin{align*}
|E(G_3)| & \le y_{\{1,2,3\}}\cdot\min\{n-\Delta-3-|X_1\cup X_2\cup X_3|+3, n-\Delta-3\}\\ 
 &+\sum\limits_{1\le i<j\le 3}y_{\{i,j\}}\cdot\min\{n-\Delta-3-|X_i\cup X_j|+2, n-\Delta-3\}\\
 &+\sum\limits_{i\in[3]}y_{\{i\}}\cdot\min\{n-\Delta-3-|X_i|+1, n-\Delta-3\}\\
 &+(\Delta-\sum\limits_{A\subseteq [3], A\neq \varnothing}y_A)(n-\Delta-3)+|X_1|+|X_2|+|X_3|+\ell\\
 &= y_{\{1,2,3\}}\cdot\min\{-|X_1\cup X_2\cup X_3|+3, 0\}
 +\sum\limits_{1\le i<j\le 3}y_{\{i,j\}}\cdot\min\{-|X_i\cup X_j|+2, 0\}\\
 &+\sum\limits_{i\in[3]}y_{\{i\}}\cdot\min\{-|X_i|+1, 0\}+\Delta(n-\Delta-3)+|X_1|+|X_2|+|X_3|+\ell
\end{align*}
The other inequality also holds\added{ similarly}, by \deleted{similar }counting the number of edges incident to the vertices of $X$ \replaced{at first}{first}.
\end{proof}

We consider cases based on the parity of $n$. We first treat the even case, where $\Delta \ge \frac{n-2}{2}$, and then the odd case.

\noindent\textbf{Case 1.} $|X_3|\ge 1$ and $2|n$.
In this case we will show that $|X_2| = |X_3| = 1$ and 
$X_2 = X_3 = \{a\}$.

Suppose not, then there exists a matching of size three between $V$ and $X$,
since by~\eqref{Eq:X_1_large} we have $|X_1| \ge 3$ whenever $n>16$.
Note that a vertex matched to $v_i$ contributes either to 
$x_{\{i\}}$, $x_{\{i,j\}}$, or $x_{\{1,2,3\}}$.
Each such contribution produces a negative term involving
$|Y_i|$ in Claim~\ref{Claim:2.1}.
Indeed,
\begin{align*}
|E(G_3)|
&\le x_{\{1,2,3\}}\!\cdot\!\min\{-|Y_1\!\cup\! Y_2\!\cup\! Y_3|+3,0\}
   + \sum_{1\le i<j\le 3} x_{\{i,j\}}\!\cdot\!\min\{-|Y_i\!\cup\! Y_j|+2,0\} \\
&\quad + \sum_{i\in[3]} x_{\{i\}}\!\cdot\!\min\{-|Y_i|+1,0\}
   + \Delta(n-\Delta-3)
   + |Y_1|+|Y_2|+|Y_3|
   + \ell \\
&\le -|Y_1|+3 -|Y_2|+3 -|Y_3|+3
   + \Delta(n-\Delta-3)
   + |Y_1|+|Y_2|+|Y_3|
   + \ell \\
&\le \Delta(n-\Delta-3)+11
   < \frac{(n-2)^2}{4},
\end{align*}
a contradiction for $n \ge 24$.
Hence we may \replaced{asusme}{assume} that we have $|X_2| = |X_3| = 1$ and 
$X_2 = X_3 = \{a\}$.

If the vertex $a$ is also adjacent to $v_1$ in $G_3$ along with $v_2$ and $v_3$, then $\ell=0$,
since $G_3$ is triangle-free.
By the \replaced{lowerbound}{lower bound} on $\Delta$ we have
\[
|X|=n-\Delta-3 \le \frac{n-4}{2},
\]
and hence
\[
|X_1|\le |X|-1 \le \frac{n-6}{2},
\]
because $X_1\subseteq X \setminus \{v\}$.
Therefore,
\[
|X_1|+|X_2|+|X_3|
\le \frac{n-6}{2}+1+1
< \frac{n}{2},
\]
contradicting~\eqref{Eq:X_1_large}.
Thus, we may assume that $v_1a \notin E(G_3)$, \replaced{thus}{so} $|X_1|\leq \frac{n-8}{2}$.
Furthermore, repeating the above argument, we have
$\ell = 2$ and $|X_1| = \frac{n-8}{2}$.
In particular\added{,} as $v_2v_3\notin E(G_3)$\added{,} we have $v_1v_3,v_1v_2\in E(G_3)$ and $X_1=X\setminus \{v,a\}$.     

The shortest $v$–$u$ path in $G$ is $v v_i v_j u$,
where $\{i,j\}=\{2,3\}$.
Hence
\[
N_G(v_i) \subseteq X \cup \{v_j\}
\quad\text{and}\quad
N_G(v_j) \subseteq Y \cup \{v_i\}.
\]
Now consider a shortest $v_1$–$v_i$ path in $G$,
say $v_1 a_1 a_2 v_i$.
Then $a_2 \in X$, and since
$a_2 v_i, a_2 v_1 \notin E(G_3)$,
it follows that $a_2 = v$.
Consequently $a_1 \in X$, and again
$a_1 v_i, a_1 v_1 \notin E(G_3)$\added{,} it follows that $a_1 = v$.
But $a_1 \neq a_2 = v$, yielding a contradiction.

\vspace{0.5em}
\noindent\textbf{Case 2.} $|X_3|= 0$ and $2|n$.
We have $|X_2| \ge 1$, since
\[
|X_1|+|X_2|+2 \ge \frac{n}{2}
\quad \text{and} \quad
|X_1|\le \frac{n-6}{2}.
\]
If $v_1v_2 \in E(G_3)$, then $v_1$ and $v_2$ have no common neighbour in $G_3$,
so $|X_1|+|X_2|\le |X|-1$, and hence
\[
\ell+|X_1|+|X_2|+|X_3|
\le 2+\frac{n-6}{2}
< \frac{n}{2},
\]
contradicting~\eqref{Eq:X_1_large}.
Thus $v_1v_2 \notin E(G_3)$.

\medskip
\noindent\textbf{Case 2.1.}
$|Y_2|\ge 1$ and $|Y_1|=0$
\quad (or symmetrically $|Y_1|\ge 1$ and $|Y_2|=0$).

If $\ell=2$, then $v_3v_2, v_3v_1 \in E(G_3)$.
Hence $Y_2\cap Y_3=\emptyset$, and therefore
$y_{\{2\}}\ge 1$.
By Claim~\ref{Claim:2.1},
\begin{align*}
    |E(G_3)|
&\le -|X_2|+1+\Delta(n-\Delta-3)
   +|X_1|+|X_2|+|X_3|+2
\le \Delta(n-\Delta-3)+|X_1|+3\\
&\le  \Delta(n-\Delta-3)+n-\Delta-1\le  (\Delta+1)(n-\Delta-3)+2\le \floor{\frac{n}{2}\frac{n-4}{2}}+2\leq \floor{ \frac{(n-2)^2}{4}}+1 
\end{align*}
Using the lower bounds on $\Delta$ and $|E(G_3)|$ together with
AM--GM, we obtain $|X_1|=\frac{n-6}{2}.$ Thus $X_1=X\setminus \{v\}$.

The shortest $v$–$u$ path is either $vv_1v_2u$
or $vv_2v_1u$.
Assume it is $vv_1v_2u$.
Let $v_3a_1a_2v_1$ be a shortest $v_3$–$v_1$ path.
Then $a_2 = v$, since $a_2 \notin X\setminus\{v\}$,
as all these vertices are at distance three from $v_1$,
$a_2 \notin Y$, as the distance from $v$ to every vertex of $Y$ is three,
and $a_2 \neq v_2$, since $d_G(v_3,a_2)\le 2$.
Hence $a_1 \notin Y$, a contradiction.

So the only possible shortest $v$–$u$ path is $vv_2v_1u$. Therefore $N_G(v_2)=\{v,v_1\},$ since $X_1 = X\setminus\{v\}$.
As $|X_2|\geq 1$ and $X_1 = X\setminus\{v\}$, we have $w_1\in X_1\cap X_2$.
The shortest path between $v_2$ and $w_1$ is
$v_2va_1w_1$ for some $a_1\in X$,
and the shortest path between $v_1$ and $w_1$ is
$v_1y_1y_2w_1$
for some $y_1,y_2\in Y\setminus\{u\}$.
Then $d_G(v_2,y_2)=3$,
since neither $v$ nor $v_1$ is adjacent to $y_2$ in $G$.
Moreover, $a_1v_2, a_1y_2 \notin E(G_3)$.
The vertex $y_2$ is adjacent in $G_3$ to exactly one vertex of $V$, namely $v_2$.
Since $a_1\in X\setminus X_2$ is not adjacent to $y_2$ in $G_3$, we obtain
\[
d_{G_3}(y_2)\le (n-\Delta-3)+1-|X_2|-1 = n-\Delta-3-|X_2|.
\]
Thus, the initial upper bound on $|E(G_3)|$ improves by one, and the desired inequality follows.

If $\ell\le 1$\deleted{,} and $\mreplaced{y_{2}}{y_{\{2\}}}\geq 1$ then\added{,} by Claim~\ref{Claim:2.1},
\begin{align*}
    |E(G_3)|
&\le -|X_2|+1+\Delta(n-\Delta-3)
   +|X_1|+|X_2|+|X_3|+1
\le \Delta(n-\Delta-3)+|X_1|+2\\
&\le  \Delta(n-\Delta-3)+n-\Delta-2\le  (\Delta+1)(n-\Delta-3)+1\le \floor{\frac{n}{2}\cdot\frac{n-4}{2}}+1\leq \floor{ \frac{(n-2)^2}{4}} 
\end{align*}
we are done. 
Since $Y_2$ is not empty, we have $Y_2 \subseteq Y_3$. \added{Hence we may assume $y_{\{2\}}=0$. }If $\ell\le 1$\deleted{,} and $\mreplaced{y_{2}}{y_{\{2,3\}}}\geq 1$ then\added{,} by Claim~\ref{Claim:2.1},
\begin{align*}
    |E(G_3)|
&\le -|X_2|+2+\Delta(n-\Delta-3)
   +|X_1|+|X_2|+|X_3|+1
\le \Delta(n-\Delta-3)+|X_1|+3\\
&\le  \Delta(n-\Delta-3)+n-\Delta-1\le  (\Delta+1)(n-\Delta-3)+2\le \floor{\frac{n}{2}\cdot\frac{n-4}{2}}+2\leq \floor{ \frac{(n-2)^2}{4}}+1 
\end{align*}
Hence we have $\Delta=\frac{n-2}{2}$, $|X_1|=\frac{n-6}{2}$, $X_1=X\setminus \{v\}$, $\ell=1$, and $Y_2\subseteq Y_3$.
Thus $v_2v_3\notin E(G_3)$ and $v_1v_3\in E(G_3)$.

Moreover, by Claim~\ref{Claim:2.1} we have
\begin{equation}\label{Eq:c2.2.2.l=1}    
|E(G_3)| \le -|Y_2|+2+\Delta(n-\Delta-3)+|Y_2|+|Y_3|+1\leq \Delta(n-\Delta-3)+|Y_3|+3.
\end{equation}
Using the lower bounds on $\Delta$ and $|E(G_3)|$ together with
AM--GM, we obtain  $|Y_3|\ge \frac{n-6}{2}$.

There are four possible shortest $v$–$u$ paths.

(i) $v v_1 v_2 u$.
Then $N_G(v)=\{v_1\}$ and $N_G(v_1)=\{v,v_2\}$ since $X_1=X\setminus \{v\}$.
Hence every vertex of $Y$ is adjacent to $v_2$ as $N_{G_3}(v)=Y$, so $|Y_2|=0$, a contradiction.

(ii) $vv_2v_3u$.
Consider a shortest $v_1$–$v_3$ path $v_1a_1a_2v_3$.
As $d_G(u,a_2)<3$ we have $a_2\in Y$ or $a_2=v_2$.
If $a_2=v_2$, then $a_1\in X$ so $a_1=v$.
Therefore we have $N_G(v_1)=\{v\}$ and $N_G(v)=\{v_1,v_2\}$ since  $X_1=X\setminus \{v\}$.
The vertex $v_3$ is adjacent in $G$ to every vertex of $Y$,
so $|Y_3|=0$, a contradiction.
Thus $a_2\in Y$, hence $a_1\in Y$.
As $|Y_3|\ge\frac{n-6}{2}$ and $a_1,a_2\notin Y_3$, we have $a_2=u$.
We get $d_G(v,a_1)=4\neq 3$\added{,} a contradiction.

(iii) $vv_2v_1u$.
Then $N_G(v_2)=\{v,v_1\}$ as $N_G(v_1)=X\setminus \{v\}$.
Let $x\in X_2$ and consider a shortest $v_2$–$x$ path $v_2a_1a_2x$.
Then $a_1=v$ and $a_2$ is either $v_3$ or lies in $X\setminus\{x,v\}$.
Next consider a shortest $v_1$–$x$ path $v_1b_1b_2x$ with $b_1,b_2\in Y\setminus\{u\}$.
Then $d_G(v_2,b_2)=3$ and $d_G(a_2,b_2)=2$.
Since $Y_2\subseteq Y_3$, we have $a_2\neq v_3$.
Now consider shortest $a_2$--$u$ path $a_2v_3c_1u$,
with $c_1\neq b_2$ since $d_G(b_2,v_3)=3$.
Hence
\[
N_{G_3}(a_2)\subseteq (Y\setminus\{b_2,c_1\})\cup\{v_1\},
\]
so $d_{G_3}(a_2)\le |Y|-1$. 
Thus, in~\eqref{Eq:c2.2.2.l=1} we counted $|Y|$ neighbours for 
$a_2\in X$ instead of $|Y|-1$.
Since $a_2$ is not at distance three from $v_2$, the bound in~\eqref{Eq:c2.2.2.l=1}  improves by one and we get
\[
|Y_3|\ge \frac{n-4}{2},
\]
a contradiction, as $|Y|=\frac{n-2}{2}$ and $v_3$ is not adjacent to $c_1$ and $u$.

(iv) $v v_3v_2u$.
If $X \subseteq N_G(v_3)$, let $y'\in Y_3$.
Then $y'$ is \replaced{at the distance of three}{at distance three} from both $v$ and $v_3$.
Consider a shortest path $y'a_1a_2v$.
We have $a_1\notin X$, since $X\subseteq N_G(v_3)$,
and $a_1\notin Y\cup V$.
This is impossible.

If $X \not\subseteq N_G(v_3)$ then there is $a_1\in X\setminus \{v\}$ such that $uy_1y_2a_1$ is a shortest $u$–$a_1$ path \deleted{with }\added{and }$y_1\neq v_2$. Then $y_1,y_2\neq v_1$ as $X_1=X\setminus \{v\}$, so $y_1,y_2\in Y$.
Since the distance \added{from }$y_1$ to $v$ is three and $y_1$ is not \replaced{incident with}{adjacent to} vertices \replaced{from}{in} $X$, we have $y_1v_2\in E(G)$.

If $d_{G_3}(y_1)<n-\Delta-3$, then since $y_1\notin Y_1\cup Y_2 \cup Y_3$  by Claim~\ref{Claim:2.1},
\[
    |E(G_3)|\le -|X_2|+2+\Delta(n-\Delta-3)-1
   +|X_1|+|X_2|+|X_3|+1\leq \floor{ \frac{(n-2)^2}{4}}. 
\]
If $d_{G_3}(y_1)=n-\Delta-3$, then replacing $u$ by $y_1$ and $y_1$ by $y_2$
and repeating the same argument, we obtain $y_2v_2\in E(G)$,
contradicting $|Y_3|\ge \frac{n-6}{2}$.

\noindent\textbf{Case 2.2.} $|Y_2|=|Y_1|=0$.
 In this case the graphs $G_3[X\cup \{v_3\}]$ and 
$G_3[Y\cup \{v_1,v_2\}]$ are empty, since $v_1v_2\notin E(G_3)$
and $|X_3|=0$.
Hence
\[
|E(G_3)|
\le \Delta\,|X\cup \{v_3\}|
= \Delta(n-\Delta-2)
\le \frac{(n-2)^2}{4},
\]
a contradiction.

\noindent\textbf{Case 2.3.} $|Y_2|\ge 1$ and $|Y_1|\ge 1$.
In this case, we have $|Y_1|=|Y_2|=1$ and $Y_1=Y_2=\{y_1\}$. 
Since otherwise, there are two distinct vertices \replaced{incident with}{adjacent to} $v_1$ and $v_2$, and by Claim~\ref{Claim:2.1} we have
\begin{align*}
|E(G_3)| & \le  y_{\{1,2,3\}}\cdot\min\{-|X_1\cup X_2\cup X_3|+3, 0\}
 +\sum\limits_{1\le i<j\le 3}y_{\{i,j\}}\cdot\min\{-|X_i\cup X_j|+2, 0\}\\
 &+\sum\limits_{i\in[3]}y_{\{i\}}\cdot\min\{-|X_i|+1, 0\}+\Delta(n-\Delta-3)+|X_1|+|X_2|+|X_3|+\ell\\
&\leq -|X_1|+3-|X_2|+3+ \Delta(n-\Delta-3)+|X_1|+|X_2|+|X_3|+\ell,
\end{align*}
a contradiction as $n\geq 18$.

If the vertex $y_1$ is also adjacent to $v_3$ in $G_3$ along with $v_1$ and $v_2$, then $\ell=0$,
since $G_3$ is triangle-free. By Claim~\ref{Claim:2.1} we have
\begin{align*}
|E(G_3)| & \le  y_{\{1,2,3\}}\cdot\min\{-|X_1\cup X_2\cup X_3|+3, 0\}
 +\sum\limits_{1\le i<j\le 3}y_{\{i,j\}}\cdot\min\{-|X_i\cup X_j|+2, 0\}\\
 &+\sum\limits_{i\in[3]}y_{\{i\}}\cdot\min\{-|X_i|+1, 0\}+\Delta(n-\Delta-3)+|X_1|+|X_2|+|X_3|+\ell\\
&\leq -|X_1\cup X_2|+3+ \Delta(n-\Delta-3)+|X_1|+|X_2|\leq \Delta(n-\Delta-3)+|X_2|+3\\
&\leq \Delta(n-\Delta-3)+n-3-\Delta-1+3\leq  (\Delta+1)(n-\Delta-3)+2 \\
&\leq \frac{n}{2}\cdot\frac{n-4}{2}+2=\frac{(n-2)^2}{4}+1.
\end{align*}
Thus, either we are done or $X_1=X_2=X\setminus\{v\}$. 

Let $vv_i v_j u$ be a shortest $v$–$u$ path in $G$ with
$\{i,j,k\}=[3]$.
Let $va_1a_2y_1$ be a shortest $v$–$y_1$ path.
Since vertices of $Y$ are at distance three from $v$,
and vertices of $V$ are at distance three from $y_1$,
it follows that $a_1 \in X$.

Then $d_{G}(a_1,v_i)\le 2$ via the path $a_1 v v_i$.
Moreover, since $d_G(a_1,u)=3$ and $a_1$ has no neighbour in $Y$,
it must have a neighbour in $V$.
If this neighbour is $v_i$ or $v_j$, then the distance
from $a_1$ to both vertices is at most two, as $v_i$ and $v_j$ are adjacent.
Otherwise, $a_1$ is adjacent to $v_k$,
and hence its distance to $v_k$ and $v_i$
is at most two,
contradicting the assumption that $v_1$ and $v_2$
are \replaced{distance three to}{at distance three from} all vertices of $X\setminus\{v\}$.
So we may assume $v_3y_1\notin E(G_3)$. 

If $\ell=0$\added{,} then \replaced{considering}{since} $y_1v_3\notin \mreplaced{E(G)}{E(G_3)}$ we have   
\begin{align*}
|E(G_3)| & \le  y_{\{1,2,3\}}\cdot\min\{-|X_1\cup X_2\cup X_3|+3, 0\}
 +\sum\limits_{1\le i<j\le 3}y_{\{i,j\}}\cdot\min\{-|X_i\cup X_j|+2, 0\}\\
 &+\sum\limits_{i\in[3]}y_{\{i\}}\cdot\min\{-|X_i|+1, 0\}+\Delta(n-\Delta-3)+|X_1|+|X_2|+|X_3|+\ell\\
&\leq -|X_1\cup X_2|+2+ \Delta(n-\Delta-3)+|X_1|+|X_2|\leq \Delta(n-\Delta-3)+|X_2|+2   \\
&\leq \Delta(n-\Delta-3)+n-3-\Delta-1+2\leq  (\Delta+1)(n-\Delta-3)+1 \\
&\leq \frac{n}{2}\cdot\frac{n-4}{2}+1=\frac{(n-2)^2}{4}.
\end{align*}

If $\ell=2$, then $v_1v_3,v_2v_3\in E(G_3)$. 
Thus $\{i,j\}=[2]$ as  $vv_iv_ju$ is a path in $G$.
By Claim~\ref{Claim:2.1} we have
\[
|E(G_3)|  \le -(|X_1\cup X_2|-2) +\Delta(n-\Delta-3)+|X_1|+|X_2|+2\leq \Delta(n-\Delta-3)+|X_2|+4.
\]
Thus we have  $|X_1|\ge |X_2|\ge \frac{n-8}{2}$, by the lower bounds for  $\Delta$ and $|E(G_3)| $ together with  AM--GM inequality.

\replaced{Simiraly}{Similarly} by Claim~\ref{Claim:2.1} we have
\[
|E(G_3)|\leq \Delta(n-\Delta-3)+1+1+|Y_3|+2
\]
Thus we have  $|Y_3|\ge \frac{n-8}{2}$, by the lower bounds for  $\Delta$ and $|E(G_3)| $ together with  AM--GM inequality. 

Consider the paths $v_i a_1 a_2 v_3$ and $v_j b_1 b_2 v_3$.
Since $a_1 \notin V$ and $d_G(v,a_1)\le 2$, we have $a_1 \notin Y$,
and hence $a_1 \in X$.
Similarly, as $b_1 \notin V$ and $d_G(u,b_1)\le 2$, we have $b_1 \notin X$,
and therefore $b_1 \in Y$.
If $a_1 = v$, then $a_2 \in X$.
Moreover, $N_G(v_i)=\{v,v_j\}$,
since $a_2$ is at distance two from $v_i$,
all other vertices of $X$ are at distance three
(because $|X_1|\ge \frac{n-8}{2}$),
and the vertices of $Y$ are at distance at least two.
Consequently,
\[
d_G(v_i,b_2)=3
\quad\text{and}\quad
d_G(v_j,b_2)=2,
\]
contradicting the assumption that $Y_2=Y_1$.

If $a_1\neq v$ then as $|X_1|\ge\frac{n-8}{2}$ there are at most two vertices not in distance three in $X$, namely $v$ and $a_1$, hence $a_2\notin X$ and as $a_2\notin V$ we have $a_2\in Y$ and $a_2\neq y_1$. So we have $|X|=\frac{n-4}{2}$ and thus $|Y|=\frac{n-2}{2}$.

Since $|Y_3| \ge \frac{n-8}{2}$ and $|Y| = \frac{n-2}{2}$,
it follows that $b_1 \in \{y_1,u,a_2\},$
and hence $b_1 = u$.
As $d_G(u,a_1)=3$, we must have $b_2 \ne a_2$.
Therefore, $b_2 \in Y \setminus \{y_1,u,a_2\}.$
However, this contradicts $|Y_3| \ge \frac{n-8}{2}$,
since there are at least four vertices \replaced{at the distance of less than three}{at distance less than three},
while $|Y|=\frac{n-2}{2}$.

If $\ell=1$, then
\[
|E(G_3)|
\le -(|X_1\cup X_2|-2)
   +\Delta(n-\Delta-3)
   +|X_1|+|X_2|+1
\le \Delta(n-\Delta-3)+|X_2|+3.
\]
Using the lower bounds on $\Delta$ and $|E(G_3)|$ together with
the AM--GM inequality, we obtain $|X_1|\ge |X_2|\ge \frac{n-6}{2}.$
Thus $X_1=X_2=X\setminus\{v\}$ and
$Y_3=Y\setminus\{y_1,u\}$.

Since $G_3$ is triangle-free, $v_1v_2\notin E(G_3)$.
As $X_1=X_2$ and $Y_1=Y_2=\{y_1\}$,
we may assume without loss of generality that
$v_3v_1\in E(G_3)$.
The shortest $v$–$u$ path must be one of the following four cases.

(1) $vv_1v_2u$.
Then $N_G(v_1)=\{v_2,v\}$ and
$N_G(v)=\{v_1\}$, since $X_1=X\setminus\{v\}$.
Hence every vertex of $Y$ is adjacent to $v_2$,
contradicting $Y_2=\{y_1\}$.

(2) $vv_2v_1u$.
Then $N_G(v_2)=\{v_1,v\}$ and
$N_G(v)\subseteq\{v_2,v_3\}$.
Let $y_1a_1a_2v_1$ be a shortest $y_1$–$v_1$ path.
Since $a_2\neq v_2$, we have $a_1,a_2\in Y$.
Thus $d_G(a_1,v_2)=3$ and $d_G(a_1,v_1)=2$,
a contradiction.

(3) $vv_2v_3u$.
Then $N_G(v_2)=\{v_3,v\}$ and
$N_G(v)\subseteq\{v_2,v_1\}$.
Since every vertex of $Y$ is at distance three from $v$,
it follows that $v_3$ is adjacent to all vertices of $Y$,
a contradiction.

(4) $vv_3v_2u$.
Then $N_G(v_2)=\{u,v_3\}$ and
$N_G(v_3)=\{v,v_2\}$.
For any $w_1\in X\setminus\{v\}$,
the only shortest path between $w_1$ and $v_2$ is
$v_2v_3vw_1$.
Hence $v$ is adjacent to all vertices of $X\setminus\{v\}$.
Thus $N_G[X\setminus\{v\}]=X$ and
$d_G(v_1,a)\ge4$ for all $a\in X$,
a contradiction.

We have finished the proof in the case where $n$ is even, and turn to the case where $n$ is odd.

First, suppose that $\Delta \ge \frac{n-1}{2}$. Consider the graph $G'$ obtained from $G$ by duplicating a vertex of degree $\Delta$ in $G_3$. It is straightforward to verify that this operation does not change the distances between any pair of vertices. In particular, $G'$ has one more vertex than $G$, and hence an even number of vertices. Therefore, the previously established bound for the even case applies to $G'$.
Consequently, we obtain
\[
|E(G_3)| = \mreplaced{|E(G')|}{|E((G')_3)|} - \Delta \le \frac{(n-1)^2}{4} - \frac{n-1}{2} = \left\lfloor \frac{(n-2)^2}{4} \right\rfloor.
\]
It follows that we may assume $\Delta = \frac{n-3}{2}$ and $|E(G_3)|\geq \frac{n^2-4n+7}{4}$. In particular, for the vertices $u$ and $v$, we have
\[
d_{G_3}(v) = d_{G_3}(u) = \frac{n-3}{2}.
\]

Therefore $|E(G_3)|-\Delta(n-3-\Delta)\geq  \frac{n^2-4n+7}{4}- \frac{n^2-6n+9}{4}=\frac{n-1}{2}$, and  by Claim~\ref{Claim:2.1}, we get
\begin{equation}\label{Eq:Claim_1_odd_1}
\begin{split}
\frac{n-1}{2} & \le  y_{\{1,2,3\}}\cdot\min\{-|X_1\cup X_2\cup X_3|+3, 0\}
 +\sum\limits_{1\le i<j\le 3}y_{\{i,j\}}\cdot\min\{-|X_i\cup X_j|+2, 0\}\\
&+\sum\limits_{i\in[3]}y_{\{i\}}\cdot\min\{-|X_i|+1, 0\}+|X_1|+|X_2|+|X_3|+\ell.
\end{split}
\end{equation}
and
\begin{equation}\label{Eq:Claim_1_odd_2}
\begin{split}
\frac{n-1}{2} & \le x_{\{1,2,3\}}\cdot\min\{-|Y_1\cup Y_2\cup Y_3|+3, 0\}
+\sum\limits_{1\le i<j\le 3}x_{\{i,j\}}\cdot\min\{-|Y_i\cup Y_j|+2,0 \}\\
&+\sum\limits_{i\in[3]}x_{\{i\}}\cdot\min\{-|Y_i|+1, 0\}+|Y_1|+|Y_2|+|Y_3|+\ell.
\end{split}
\end{equation}

\begin{claim}\label{odd1}
For all distinct vertices $w,w'\in V(G)$, we have
\[
N_G(w)\setminus \{w'\}\not\subseteq N_G(w').
\]
\end{claim}

\begin{proof}
Suppose there exist distinct vertices $w,w'\in V(G)$ such that
$N_G(w)\setminus \{w'\}\subseteq N_G(w').$
Delete the vertex $w$ from $G$. In the graph $G_3$, this operation removes at most $\frac{n-3}{2}$ edges, since every vertex has degree at most $\frac{n-3}{2}$ in $G_3$ and $N_G(w)\setminus \{w'\}\subseteq N_G(w')$.

Since $n-1$ is even, the extremal bound for graphs on $n-1$ vertices has already been established. Hence,
\[
|E(G_3)|\le \frac{(n-3)^2}{4}+\frac{n-3}{2}
= \frac{(n-1)(n-3)}{4},
\]
which yields the desired upper bound.
\end{proof}
\begin{claim}\label{odd2}
Let $vw_1w_2u$ be a shortest $v$--$u$ path, and let $w_3$ be the unique vertex in $\{v_1,v_2,v_3\}\setminus \{w_1,w_2\}$. 
Then $vw_3,uw_3\notin E(G)$, and there exist vertices $v'\in N_G(v)\setminus N_G[w_1]$ and $u'\in N_G(u)\setminus N_G[w_2]$ such that $v'w_3,u'w_3\in E(G)$. 
Moreover, if $|X_i|=\frac{n-5}{2}$, then $v_i=w_2$, and if $|Y_i|=\frac{n-5}{2}$, then $v_i=w_1$.
\end{claim}

\begin{proof}
Suppose first that $vw_3\in E(G)$. Since $vw_1w_2u$ is a shortest $v$--$u$ path, it follows that $uw_3\notin E(G)$. By Claim~\ref{odd1}, there exists a vertex $u'\in N_G(u)\setminus N_G[w_2]$. In particular, $u'w_2\notin E(G)$. As $uw_3\notin E(G)$ and $u'\in N_G(u)$, we also have $u'w_1,~u'w_3\notin E(G)$ as $d(v,u')=3$. Consequently, every path from $u'$ to $v$ must have length at least $4$, a contradiction. Hence $vw_3\notin E(G)$. By symmetry, $uw_3\notin E(G)$.

Let $v'\in N_G(v)\setminus N_G[w_1]$ and $u'\in N_G(u)\setminus N_G[w_2]$, whose existence follows from Claim~\ref{odd1}. Since $vw_3,uw_3\notin E(G)$, it follows that $v'\in X$ and $u'\in Y$. 
As the distance is three from $v'$ to $u$, and $u'$ to $v$ we have $v'w_3,u'w_3\in E(G)$.

Finally, suppose that $|X_i|=\frac{n-5}{2}$. 
As $v'w_1,~v'w_3,~vw_1,~vw_3\notin E(G_3)$, and $\abs{X}=\frac{n-3}{2}$, which forces $v_i=w_2$. 
The statement for $|Y_i|=\frac{n-5}{2}$ follows analogously.
\end{proof}

Recall that we assume $|X_1|\ge |X_2|\ge |X_3|$.

\noindent\textbf{Case 3.} $|X_3|\ge 1$ and $2|n-1$. 
As in Case 1\deleted{.}\added{,} we have $|X_2|=|X_3|=1$ and $X_2=X_3=\{a\}$.
Assume first that $|Y_1|\ge 1$. By~\eqref{Eq:Claim_1_odd_1},
\begin{equation*}
\begin{split}
\frac{n-1}{2} & \le  y_{\{1,2,3\}}\cdot\min\{-|X_1\cup X_2\cup X_3|+3, 0\}
 +\sum\limits_{1\le i<j\le 3}y_{\{i,j\}}\cdot\min\{-|X_i\cup X_j|+2, 0\}\\
&+\sum\limits_{i\in[3]}y_{\{i\}}\cdot\min\{-|X_i|+1, 0\}+|X_1|+|X_2|+|X_3|+\ell \\
&\leq -|X_1|+3+|X_1|+|X_2|+|X_3|+\ell \leq  7
\end{split}
\end{equation*}
a contradiction since $n>15$. Hence we have $|Y_1|=0$.

Assume next that $\ell=0$. By~\eqref{Eq:Claim_1_odd_1}, we have $|X_1|=\frac{n-5}{2}$. 
By Claim~\ref{odd2}, we have $v_1=w_2$ and $w_1$ has no neighbor in $X\setminus\{v\}$. Since $d(w_1,w_3)\le 3$, and $\ell=0$ we have $d(w_1,w_3)\leq 2$ and  one of the following holds: (i) $w_1w_3\in E(G)$, (ii) $w_2w_3\in E(G)$, or (iii) $w_1$ and $w_3$ have a common neighbor in $X$. The cases (ii) and (iii) are impossible as $|X_1|=\frac{n-5}{2}$. Thus $w_1w_3\in E(G)$.

Applying \eqref{Eq:Claim_1_odd_2}, we obtain
\[
-|Y_2\cup Y_3|+3+|Y_2|+|Y_3|\ge \frac{n-1}{2}.
\]
Since $w_1w_3\in E(G)$, we have $|Y_2|,|Y_3|\le \frac{n-7}{2}$, and hence
\[
|Y_2|=|Y_3|=\frac{n-7}{2}\quad\text{and}\quad Y_2=Y_3=Y\setminus \{u,u'\}.
\]
It follows that $N_{G_3}(w_1)\cap Y=Y\setminus\{u,u'\}$, and therefore $u$ is adjacent to all vertices in $Y\setminus\{u\}$. Let $u''\in Y\setminus\{u,u'\}$. Then $u''w_2\in E(G)$ or $u''w_3\in E(G)$, contradicting $|Y_2|=|Y_3|=\frac{n-7}{2}$.

Hence $\ell\ge 1$, and moreover $av_1\notin E(G_3)$. As $G_3$ is triangle-free and $av_2,~av_3\in E(G_3)$, we have $v_2v_3\notin E(G_3)$. 
If $\ell=2$, then $\{w_1,w_2\}=\{v_2,v_3\}$. 
We have $d(w_2,v')=3$ since $w_2w_3\notin E(G_3)$.
As $d(w_1,v')=2$ and $\{w_1,w_2\}=\{v_2,v_3\}$ we get $X_2\neq X_3$, a contradiction.
Thus, we have $\ell=1$. 

By~\eqref{Eq:Claim_1_odd_1},
\[
|X_1|+|X_2|+|X_3|+1\ge \frac{n-1}{2},
\]
and hence $X_1=X\setminus\{v,a\}$, $|X_1|=\frac{n-7}{2}$. 
We have $x_{\{1,2,3\}}=0$ and by \eqref{Eq:Claim_1_odd_2},
\[
-|Y_2\cup Y_3|+2+|Y_2|+|Y_3|\ge \frac{n-1}{2},
\]
which implies $|Y_2|,|Y_3|\ge \frac{n-7}{2}$.

Since $\ell=1$, we have $w_1w_3,w_2w_3\notin E(G)$, and thus either $w_2$ and $w_3$ have a common neighbor in $Y$, or $w_1$ and $w_3$ have a common neighbor in $X$.

If $w_2$ and $w_3$ have a common neighbor $u''$ in $Y$, then we have $N_{G_3}(w_2)\cap Y,\,N_{G_3}(w_3)\cap Y\subseteq  Y\setminus \{u,u',u''\},$
implying $|Y_2|\le \frac{n-9}{2}$ or $|Y_3|\le \frac{n-9}{2}$, a contradiction.

Hence $w_1$ and $w_3$ have a common neighbor $v''\in X$. Since $|X_1|=\frac{n-7}{2}$, we have $v_1=w_2$. Moreover, $N_{G_3}(w_3)\cap Y\subseteq Y\setminus\{u,u'\}$, so $|Y_2|\le \frac{n-7}{2}$ or $|Y_3|\le \frac{n-7}{2}$.
We have $v''\notin X_1\cup X_2\cup X_3$ and
\[
d_{G_3}(v'')\le \frac{n-5}{2}.
\]
Thus, compared to the estimate in Claim~\ref{Claim:2.1}, we gain one unit, since there the corresponding degree term was bounded by
\[
n-\Delta-3=\frac{n-3}{2}.
\]
Consequently, the bound in \eqref{Eq:Claim_1_odd_2} improves by one.  
\begin{equation*}
\begin{split}
\frac{n+1}{2} & \le x_{\{1,2,3\}}\cdot\min\{-|Y_1\cup Y_2\cup Y_3|+3, 0\}
+\sum\limits_{1\le i<j\le 3}x_{\{i,j\}}\cdot\min\{-|Y_i\cup Y_j|+2,0 \}\\
&+\sum\limits_{i\in[3]}x_{\{i\}}\cdot\min\{-|Y_i|+1, 0\}+|Y_1|+|Y_2|+|Y_3|+\ell\\
&\leq -|Y_2\cup Y_3|+2+|Y_2|+|Y_3|+1\leq \frac{n-1}{2}
\end{split}
\end{equation*}
since $x_{\{2,3\}}=1$, a contradiction.

\noindent\textbf{Case 4.} $|X_3|=0$ and $2|n-1$. By symmetry, we may assume that $|Y_i|=0$ for some $i\in[3]$.

Suppose first that $|Y_3|=0$. Then if $|X_2|>0$, we have $x_{\{1\}}+x_{\{1,2\}}\ge 2$ and by \eqref{Eq:Claim_1_odd_2} we have
\begin{equation*}
\begin{split}
\frac{n-1}{2} & \le x_{\{1,2,3\}}\cdot\min\{-|Y_1\cup Y_2\cup Y_3|+3, 0\}
+\sum\limits_{1\le i<j\le 3}x_{\{i,j\}}\cdot\min\{-|Y_i\cup Y_j|+2,0 \}\\
&+\sum\limits_{i\in[3]}x_{\{i\}}\cdot\min\{-|Y_i|+1, 0\}+|Y_1|+|Y_2|+|Y_3|+\ell \leq 6,
\end{split}
\end{equation*}
a contradiction as $n>13$.
Hence we have $|X_2|=0$.

By \eqref{Eq:Claim_1_odd_1}, we have $|X_1|=\frac{n-5}{2}$ and $\ell=2$. Consequently, $|Y_1|=0$ and $|Y_2|=\frac{n-5}{2}$. By Claim~\ref{odd2}, it follows that $w_1=v_2$ and $w_2=v_1$. Hence $N_G(v_2)=\{v,v_1\}$ and $N_G(v_1)=\{u,v_2\}$. It follows that $v$ is adjacent to every vertex in $X\setminus\{v\}$ and $u$ is adjacent to every vertex in $Y\setminus\{u\}$.
Thus every vertex in $X\setminus\{v\}$ and in $Y\setminus\{u\}$ is adjacent to $v_3$, since their distances to $u$ and $v$, respectively, are equal to three in $G$. Consequently, any two such vertices are at a distance of at most two, a contradiction.

Thus we may assume that $|Y_1|=0$ (the case $|Y_2|=0$ is symmetric). By \eqref{Eq:Claim_1_odd_1},
\begin{align}\label{eq:2.1_1}
\frac{n-1}{2} & \le y_{\{2,3\}}\cdot\min\{-|X_2|+2, 0\}
+y_{\{2\}}\cdot\min\{-|X_2|+1, 0\}+|X_1|+|X_2|+\ell,
\end{align}
and by \eqref{Eq:Claim_1_odd_2}
\begin{align}\label{eq:2.1_2}
\frac{n-1}{2} & \le x_{\{1,2\}}\cdot\min\{-|Y_2|+2, 0\}
+x_{\{2\}}\cdot\min\{-|Y_2|+1, 0\}+|Y_2|+|Y_3|+\ell.
\end{align}

If $\ell=0$, then we \replaced{do not have}{cannot have} $|X_1|=|Y_3|=\frac{n-5}{2}$. Since \replaced{othervice}{otherwise}   
$v$ is adjacent to every vertex in $X\setminus\{v\}$ and $u$ is adjacent to every vertex in $Y\setminus\{u\}$.
Thus every vertex in $X\setminus\{v\}$ and in $Y\setminus\{u\}$ is adjacent to $v_3$, since their distances to $u$ and $v$, respectively, are equal to three in $G$. Consequently, any two such vertices are at a distance of at most two, a contradiction.
By \eqref{eq:2.1_1} and \eqref{eq:2.1_2} we have $|X_2|\ge 2$ and $|Y_2|\ge 2$, but since $|X_1|$ and $|Y_3|$ cannot both be equal to $\frac{n-5}{2}$, at the same time  we have $|X_2|\ge 2$ and $|Y_2|\ge 3$ (or $|X_2|\ge 3$ and $|Y_2|\ge 2$). 
By \eqref{eq:2.1_2},
\[
\frac{n-1}{2}\le 2(-|Y_2|+2)+|Y_2|+|Y_3|\le 1+|Y_3|\leq \frac{n-3}{2},
\]
a contradiction. If $|X_2|\ge 3$, then by  \eqref{eq:2.1_1} we have a contradiction.

If $\ell=1$ and $v_1v_2\in E(G_3)$ (or symmetrically $v_2v_3\in E(G_3)$), then $|X_1|+|X_2|\le \frac{n-5}{2}$. Hence, by \eqref{eq:2.1_1},
\[
\frac{n-1}{2}\le |X_1|+|X_2|+\ell\le \frac{n-5}{2}+1,
\]
a contradiction.

If $\ell=1$ and $v_1v_3\in E(G_3)$, we may assume $d_G(w_2,w_3)=3$ and $d_G(w_1,w_3)=2$, by symmetry. Thus, we have $w_1=v_2$. Since $d(v_2,u')=3$ and $d(w_2,u'),d(w_3,u')\le 2$, we obtain $y_{\{2\}}\ge 1$. By \eqref{eq:2.1_1},
\[
\frac{n-1}{2}\le -|X_2|+1+|X_1|+|X_2|+1,
\]
and hence $|X_1|=\frac{n-5}{2}$. However,
\[
|N_{G_3}(w_2)\cap X|,\ |N_{G_3}(w_3)\cap X|\le \frac{n-7}{2},
\]
a contradiction.

Finally, assume $\ell=2$. Then $v_1v_2\in E(G_3)$ or $v_2v_3\in E(G_3)$. 
Since $|X_3|=|Y_1|=0$ and the roles of $X$ and $Y$ are symmetric, we may, without loss of generality, assume that $v_1v_2\in E(G_3)$.
Since $G_3$ is triangle-free, we have $X_1\cap X_2=\emptyset$, and hence $|X_1|+|X_2|\le \frac{n-5}{2}$. Thus we have  $|X_1|+|X_2|=\frac{n-5}{2}$,  by \eqref{eq:2.1_1}. 
For each vertex in $Y\setminus Y_2$, Claim~\ref{Claim:2.1} (in \eqref{eq:2.1_1}) contributes at most $\frac{n-3}{2}$ to the upper bound on $|E(G_3)|$. Since equality holds, it follows that every such vertex has degree exactly $\frac{n-3}{2} $ in $G_3$. In particular, for every $w\in Y\setminus Y_2$, we have
\begin{equation}\label{eq:Y-y_2}
d_{G_3}(w)=\frac{n-3}{2}.    
\end{equation}

If $w_1\ne v_3$, then $\{w_1,w_3\}=\{v_1,v_2\}$, and hence $v'\notin X_1\cup X_2$, implying $|X_1|+|X_2|\le \frac{n-7}{2}$, a contradiction. Thus $w_1=v_3$.

Consider $u'$. Since $d(u',w_2), d(u',w_3)\le 2$, we have $u'\notin Y_2$, and hence $d_{G_3}(u')=\frac{n-3}{2}$ and
$N_{G_3}(u')=(X\cup\{v_3\})\setminus\{v'\}.$
Fix $x_1\in X\setminus\{v'\}$ and consider a shortest $u'$--$x_1$ path. There are three possibilities:

\medskip
\noindent\emph{(i)} $u'w_3u''x_1$ for some $u''\in Y\setminus\{u,u'\}$. Since $uu''\notin E(G_3)$, we have $d(u,u'')\le 2$, and hence there exists $y\in Y\setminus\{u,u',u''\}$ such that $uy,u''y\in E(G)$ as $d(u,u'')\le 2$. Then $d(w_2,y),d(w_3,y)\le 2$, implying $y\notin Y_2$ and thus $d_{G_3}(y)=\frac{n-3}{2}$ and $N_{G_3}(y)=(X\cup\{v_3\})\setminus\{x_1\}$. Hence, $w_2y,w_3y\notin E(G)$, which implies $d(y,v)\ge 4$, a contradiction.

\medskip
\noindent\emph{(ii)} $u'y_1y_2x_1$ for some $y_1,y_2\in Y\setminus\{u,u'\}$. Consider a shortest $v$--$y_1$ path. Then $y_1w_2\in E(G)$ or $y_1w_3\in E(G)$. If $y_1w_2\in E(G)$, then $y_1\notin Y_2$ and $w_2,w_1,w_3,x_1\notin N_{G_3}(y_1)$, a contradiction to \eqref{eq:Y-y_2}. 
Thus $y_1w_3\in E(G)$. If $d(w_2,y_1)\le 2$, then again $y_1\notin Y_2$ and $w_2,w_3,v',x_1\notin N_{G_3}(y_1)$, a contradiction to \eqref{eq:Y-y_2}. 
Hence $d(w_2,y_1)=3$, which implies $w_2=v_2$ and $w_3=v_1$. 
Since $uy_2\notin E(G_3)$, we have $d(u,y_2)\le 2$, and hence there exists $y\in Y\setminus\{u',y_1\}$ such that $yu,yy_2\in E(G)$. Then $d(v_2,y)\le 2$, so $y\notin Y_2$ and $d_{G_3}(y)=\frac{n-3}{2}$. 
Since $uy\in E(G)$ and $d_{G}(y,v)=3$, either $v_2y\in E(G)$ or $v_1y\in E(G)$. 
In the first case, $v_2,v_3,x_1\notin N_{G_3}(y)$, forcing $v_1\in N_{G_3}(y)$, a contradiction to $|Y_1|=0$. In the second case, $v_1,v',x_1,v_2\notin N_{G_3}(y)$, also a contradiction to \eqref{eq:Y-y_2}.

\medskip
\noindent\emph{(iii)} $u'w_3v'x_1$. Assume that all shortest $u'$--$x_1$ paths with $x_1\in X\setminus\{v'\}$ are of this form; otherwise, we are done by the previous cases.
Then $v'$ is adjacent to every vertex in $X\setminus\{v'\}$, while $w_3$ is adjacent to none of them. 
Consequently, $|X_2|=0$ and $|X_1|=\frac{n-5}{2}$, which implies $v_1=w_2$ and $v_2=w_3$. 
Let $v''\in X\setminus\{v,v'\}$. Since $v''v_3\notin E(G_3)$, we have $d(v'',v_3)\le 2$, and hence $vv''\in E(G)$. 
On the other hand, $v''v_3,v''v_2\notin E(G)$, which implies $d(v'',u)\ge 4$, a contradiction.
\end{proof}

\section*{Acknowledgments}
The research of He was supported by the National Natural Science Foundation of China, grant 12401445. 
The research of Salia was supported by the National Science Centre grant 2021/42/E/ST1/00193.
The author Zhu is supported by NSFC under grant 12401445, Basic Research Program of Jiangsu Province(BK20241361).

\bibliographystyle{abbrv}
\bibliography{references.bib}

@article{li2015note,
  title={Note on a {Tur\'an}-type problem on distances},
  author={Li, Xueliang and Ma, Jing and Shi, Yongtang and Yue, Jun},
  journal={Ars Comb.},
  volume={119},
  pages={211--219},
  year={2015}
}

@article{kang2005maximum,
  title={Maximum {$K_{r+1}$}-free graphs which are not $r$-partite},
  author={Kang, Mihyun and Pikhurko, O},
  journal={Matematychni studii},
  volume={24},
  number={1},
  pages={12--20},
  year={2005},
  publisher={VNTL Publishers}
}

@article{tyomkyn2013turan,
  title={A {Tur{\'a}n}-type problem on distances in graphs},
  author={Tyomkyn, Mykhaylo and Uzzell, Andrew J},
  journal={Graphs and Combinatorics},
  volume={29},
  number={6},
  pages={1927--1942},
  year={2013},
  publisher={Springer}
}

@article{bollobas2012walks,
  title={Walks and paths in trees},
  author={Bollob{\'a}s, B{\'e}la and Tyomkyn, Mykhaylo},
  journal={Journal of Graph Theory},
  volume={70},
  number={1},
  pages={54--66},
  year={2012},
  publisher={Wiley Online Library}
}

@article{csikvari2010poset,
  title={On a poset of trees},
  author={Csikv{\'a}ri, P{\'e}ter},
  journal={Combinatorica},
  volume={30},
  number={2},
  pages={125--137},
  year={2010},
  publisher={Springer}
}

@article{erdos1962theorem,
  title={On a theorem of {Rademacher}-{Tur{\'a}n}},
  author={Erd{\H{o}}s, Paul},
  journal={Illinois J. Math},
  volume={6},
  number={122-127},
  pages={1--3},
  year={1962}
}

\end{document}